\documentclass[]{amsart}
\usepackage{amssymb,amsmath,amsthm,todonotes,xfrac}
\newtheorem{theorem}{Theorem}

\newtheorem{lemma}{Lemma}
\newtheorem{cor}{Corollary}

\theoremstyle{remark}
\newtheorem{remark}{Remark}

\DeclareMathOperator{\Aut}{Aut}
\DeclareMathOperator{\Stab}{Stab}
\newcommand{\N}{\mathbb{N}}

\title{Edge Kempe equivalence of regular graph covers}
\author{Nir Lazarovich and Arie Levit}

\begin{document}

\begin{abstract} 
Let $G$ be  a finite $d$-regular graph  with a proper  edge coloring. An edge Kempe switch is a new proper edge coloring of  $G$ obtained by switching the two colors along some  bi-chromatic cycle. We prove that any other edge coloring  can be obtained by performing finitely many edge Kempe switches, provided that $G$ is replaced with a suitable finite covering graph. The required covering degree is bounded above by a constant depending only on $d$.
\end{abstract}

\maketitle

Let $G$ be a finite $d$-regular graph\footnote{Throughout this text graphs are undirected and without loops.}. Let $c:E(G)\to \{1,\ldots,d\}$ be a \emph{proper edge-coloring}, i.e every two incident edges $e, e' \in E(G)$  satisfy $c(e) \neq c(e')$. 
Given a subset  $S \subset \{1,\ldots,d\}$   of colors, let $(G,c)[S]$ denote the spanning subgraph of $G$ containing all the edges $e \in E(G)$ with $c(e) \in S$.   We will often abuse notation and write  $(G,c)[i]$ and $(G,c)[i,j]$ for the subgraphs $(G,c)[\{i\}]$ and $(G,c)[\{i,j\}]$ respectively,  where $i,j\in \{1,\ldots,d\}$. 

Observe that for any two distinct colors $i$ and $j$ the subgraph $(G,c)[i,j]$  is a collection of disjoint cycles. Its connected components are the \emph{bi-chromatic cycles of $(G,c)$ of type $i,j$}.  

Let  $\gamma$ be a bi-chromatic cycle of type $i,j$ in $G$. Let $\gamma c$ denote the   coloring of $G$ obtained from $c$ by switching the two edge colors $i$ and $j$ along $\gamma$. More precisely  the coloring $\gamma c$ is given by
\[
\gamma  c(e) := 
\begin{cases}
c(e)&\text{ if }e\notin \gamma,\\
i &\text{ if }e\in \gamma\text{ and }c(e)=j, \\
j &\text{ if }e\in \gamma\text{ and }c(e)=i.
\end{cases}
\]
for every edge $e \in E(G)$. Clearly $\gamma c$ is  proper as well. We say that the colored graph $(G,\gamma c)$ is obtained from $(G,c)$ by an \emph{edge Kempe switch operation}. Kempe introduced a similar idea in his paper \cite{kempe1879geographical} on the four color problem.

Two proper edge colorings $c_1$ and $c_2$ are \emph{edge Kempe equivalent} if $(G,c_2)$ can be obtained from $(G,c_1)$ by a sequence of edge Kempe switches. We denote this by  $(G,c_1) \sim (G,c_2) $.

Recall that a graph homomorphism $p : \overline{G} \to G$ is a \emph{covering} if $p$ is surjective and  induces a bijection between the edges incident at $v$ and $p(v)$ for every $v \in V(\overline{G})$. We say that $\overline{G}$ is a \emph{graph cover} of $G$. Note that every graph cover of $G$ must be $d$-regular as well.  The \emph{covering degree} of $p$ is well-defined whenever $G$ is connected and is equal to   $|p^{-1}(v)|$ for some $v \in V(G)$. A graph cover is \emph{finite} if the set $p^{-1}(v)$   is finite for every $v \in V(G)$.
Given a proper edge coloring $c$ of $G$  let $c \circ p$ denote the \emph{pull-back coloring} of $\overline{G}$, that is $(c \circ p)(e) = c(p(e))$ for every $e \in E(\overline{G})$. 

\begin{theorem}\label{thm:main result}
	Let $G$ be a finite $d$-regular  graph admitting two proper edge colorings  $c_1$ and $c_2$.  Then there is a finite graph cover $p:\overline{G} \to G$ so that the two pull-back colorings $c_1 \circ p$ and $ c_2 \circ p$ of $\overline{G}$ are edge Kempe equivalent. Moreover the covering degree of $p$ is bounded from above by a constant depending only on $d$.
\end{theorem}

In other words the theorem says that $(\overline{G}, c_1 \circ p) \sim (\overline{G}, c_2 \circ p)$ for some finite covering. This is false without passing to a finite cover  \cite[Lemma 4.4]{haas2014counting}. In general  the covering $\overline{G}$ may depend on the colorings $c_1$ and $c_2$. The upper bound on the degree of $p$ can be computed explicitly as a function of $d$.

\begin{remark}
Theorem \ref{thm:main result} can be seen as a generalization of the well-known fact that a symmetric group is generated by transpositions. To see this, let $G$ be the unique $d$-regular graph with two vertices and no loops. Then edge Kempe switches are simply transpositions of edges, and Theorem \ref{thm:main result} holds true without the need to pass to a cover.
\end{remark}

Let $T$ be the $d$-regular tree and   $\Aut(T)$ be its automorphism group. Let $\Gamma\le \Aut(T)$ be a uniform lattice and $C$ be the commensurator of   $\Gamma$ in the group $\Aut(T)$. For the definitions see \cite[\S 1.2, \S 6.1]{bass2001tree}. Fix an arbitrary vertex $v \in T$ and  let  $C_v = \Stab_C(v) \le C$ be the stabilizer subgroup of the vertex $v$. An element of the group $C_v$ can be interpreted as a finite $d$-regular graph with a given base point and with two proper edge colorings, see   \cite[Proposition 2.3]{lubotzky1994superrigidity}. 
This description is unique up to passing to finite covers. 
We thus obtain the following 

\begin{cor}
Edge Kempe switches are a generating set for the group $C_v$.
\end{cor}
We point out that the group $C_v$ is known to be infinitely generated, using arguments as in    \cite[Corollary B]{Bartholdi2010comm}. The above corollary provides a non-trivial such generating set.


\subsection*{Graph coverings}


 We begin with two elementary lemmas on graph coverings and colorings. In what follows, let $G$ be a finite $d$-regular graph with two proper edge colorings $c_1$ and $c_2$.

\begin{lemma}
\label{lemma:Kempe equivalence preserved under covers}
Let $p: \overline{G} \to G$ be a finite graph covering. If   $(G,c_1)\sim (G,c_2)$ then $(\overline{G}, c_1 \circ p) \sim (\overline{G}, c_2 \circ p)$.
\end{lemma}
\begin{proof}
It  suffices to deal with the case where $c_2 = \gamma c_1$ for some $c_1$-bi-chromatic cycle $\gamma$ of type $i,j$. Note that $p^{-1}(\gamma)$ is a disjoint union of the $(c_1 \circ p)$-bi-chromatic cycles $\gamma_1, \ldots, \gamma_m$ in $\overline{G}$ for some $m \in \mathbb{N}$.   We conclude by observing that 
$$ (\overline{G},\gamma_1(\ldots \gamma_m (c_1\circ p)))=(\overline{G},(\gamma c_1)\circ p)=(\overline{G},c_2\circ p).$$
\end{proof}


\begin{lemma}
	\label{lemma:can extend covers}
	Let $H\subseteq G$ be a subgraph and $p:H'\to H$ be a finite covering map. Then there exist finite coverings $q:\overline{H}\to H'$ and $r:\overline{G} \to G$ so that $\overline{H}$ is a subgraph of $\overline{G}$ and 
	$ r_{|\overline{H}} = p \circ q $.
\end{lemma}

\begin{proof} 
The degree of $p$ is well defined on every connected component of $H$. By standard covering theory \cite{hatcher2005algebraic}, any finite connected graph admits a finite covering of any degree. Applying this fact, one can construct a graph covering  $q : \overline{H} \to H'$ so that  $|(p \circ q)^{-1}(v)| = m\in \N$ for all $v\in V(H)$.  For instance $m$ can be taken to be the least common multiple of the degrees of $p$ on connected components of $H$. 

It is clear that $p\circ q : \overline{H} \to H$ can be extended to a graph covering $r : \overline{G} \to G$ by adding $m$ additional vertices and $m$ additional edges for every vertex in $V(G) \setminus V(H)$ and every edge in $E(G) \setminus E(H)$, respectively. In particular $ r_{|\overline{H}} = p \circ q $ is satisfied.
\end{proof}

\begin{remark}
\label{rem:about using lemma}
If $|p^{-1}(v)|$ is independent of $v \in V(H)$ in the previous lemma then $q$ is the identity map and the covering degree of $r$ is the same as  $|p^{-1}(v)|$ for some $v \in V(H)$. This will always be the case when we apply Lemma \ref{lemma:can extend covers} below.
\end{remark}

\subsection*{The main argument}

For every natural number $d$  let $\mathcal{K}_d$ denote the statement
$$ \mathcal{K}_d: \emph{\text{Theorem \ref{thm:main result} holds for all finite $d$-regular graphs}} $$
Clearly $\mathcal{K}_1$ and $\mathcal{K}_2$ are true.
The  strategy is to proceed by induction  assuming $\mathcal{K}_{d-1}$ to show $\mathcal{K}_{d}$. Interestingly $\mathcal{K}_{d-1}$ plays a twofold role --- first in aligning one color correctly and second in dealing with all of the remaining colors, see Corollary \ref{cor:cover aligning one color} and Lemma \ref{prop: end of proof} respectively. The proof moreover relies on a certain explicit cover of $G$ constructed below in Lemma \ref{lemma: cool cover}. Our strategy in illustrated in an example at the end of this paper, see Figure \ref{fig:example}.

We show how the induction assumption $\mathcal{K}_{d-1}$ is used to complete the proof of Theorem \ref{thm:main result} given that one color has been correctly aligned.

\begin{lemma}
	\label{prop: end of proof}	
	Assume $\mathcal{K}_{d-1}$. If   $(G, c_1 )[d] = (G, c_2)[d]$ then there is a finite covering $p:\overline{G}\to G$ such that $(\overline{G},c_1\circ p)\sim (\overline{G},c_2\circ p)$.
\end{lemma}
\begin{proof}
	Consider the spanning subgraph 
	$$H=(G, c_1 )[\{1,\ldots,d-1\}]=(G, c_2 )[\{1,\ldots,d-1\}]$$
	of $G$ consisting of those edges $e \in E(G)$ with $c_1(e) \neq d$ (or equivalently $c_2(e) \neq d$).  The graph $H$ is a finite $(d-1)$-regular graph admitting two proper colorings $c_1|_H$ and $c_2|_H$. According to $\mathcal{K}_{d-1}$ there is a finite graph cover $p:H'\to H$ in which $(H',c_1\circ p)\sim (H',c_2 \circ p)$. By Lemma \ref{lemma:can extend covers}  there exist finite coverings $r:\overline{G}\to G$ and  $q:\overline{H} \to H'$ so that $\overline{H}$ is a subgraph of $\overline{G}$ and $r_{|\overline{H}} = p \circ q$. Note that $(\overline{H},c_1\circ p\circ q)\sim (\overline{H},c_2\circ p\circ q)$ by Lemma \ref{lemma:Kempe equivalence preserved under covers}. 
	The same sequence of edge Kempe switches on $\overline{H}$ can be regarded as being performed on $\overline{G}$ implying that $(\overline{G},c_1\circ r)\sim(\overline{G},c_2\circ r)$.
\end{proof}

 We  describe an explicit construction of a special  cover making it possible to align one particular color, say  $d$.   Denote
$$G_\text{s} = (G,c_1)[d] \cap (G,c_2)[d] $$
The subgraph $G_\text{s}$ is a disjoint union of single vertices and copies of $K_2$ (that is, graphs consisting of two vertices connected by a single  edge).
Denote 
$$ G_\text{r} = \left( ( G,c_1)[d] \cup (G,c_2)[d]  \right) \setminus E (G_\text{s}) $$
The subgraph $G_\text{r}$ is a disjoint union of single vertices and of cycles.
Every vertex $v \in V(G)$ is incident  to either a single edge of $E(G_\text{s})$ or  exactly two edges of  $ E(G_\text{r})$.


\begin{lemma}\label{lemma: cool cover}
	There is a graph covering $p : \overline{G} \to G$ and a coloring $\overline{c}$ of $\overline{G}$ such that $(\overline{G},c_1 \circ p)[d] = (\overline{G},\overline{c})[d]$  and $p^{-1}(G_\text{r})$ is a disjoint union of $\overline{c}$-bi-chromatic cycles.
\end{lemma}
\begin{proof}
	Let $C = \{1, \ldots, d-1\}$ be the set of  colors excluding the color $d$. For the purpose of the proof we identify  $C$   with the additive group $\sfrac{\mathbb{Z}} {(d-1)\mathbb{Z}}$ of the integers modulo $d-1$. The graph cover $\overline{G}$ is constructed by letting
	$$V(\overline{G}) = V(G) \times C \quad \text{and} \quad E(\overline{G}) = E(G) \times C.$$
	The  map $p : \overline{G} \to G$ is simply the projection to the first coordinate. 
	
	There is a well-defined map $\rho : V(G ) \to E(G )$
	so that $\rho(v)$ is the unique edge of  $(G,c_2)[d]$ incident at the vertex $v$.
	Denote 
	$$ \kappa : V(G) \to C, \quad \kappa : v \mapsto \begin{cases}  c_1(\rho(v)), & \text{$v$ is incident to an edge of $G_\text{r}$} \\ 0  , & \text{$v$ is incident to an edge of $G_\text{s}$} \end{cases}. $$
	Fix an arbitrary orientation for the edges of $G$.  For every edge $e \in E(G)  $ let $e_+,e_- \in V(G)$ denote the origin and terminus vertices of $e$ with respect to this orientation, respectively. 
	%
	%
	Denote
	$$ \delta : E(G)  \to C, \quad \delta(e) = \begin{cases} \kappa( e_+ ) - \kappa( e_- ), & c_1(e) \neq d, \\ 0, & c_1(e) = d \end{cases} .$$
	%
	
	The edge structure of $\overline{G}$ is  as follows ---
	 the edge $\overline{e} = (e,i) \in E(\overline{G})$   is incident at the two vertices 
	$$ \overline{e}_- = (e,i)_- = (e_-,i), \quad \overline{e}_+ = (e,i)_+ = (e_+, i + \delta(e)) $$
	The definition of $\overline{G}$ is independent of the chosen orientation on $E(G)$. Indeed, given an edge $e \in E(G)$ let $e'$ denote the same edge regarded with the reverse orientation, so that $e_- = e'_+$ and $e_+ = e'_-$. In any case $\delta(e) = -\delta(e')$ which implies that if we set $\overline{e}' = (e',i+\delta(e))$ then $\overline{e}_- = \overline{e}'_+$ and $\overline{e}_+ = \overline{e}'_-$. 
	

	We are ready to define the coloring $\overline{c}$ on $\overline{G}$. It is given by
	$$ \overline{c}(e,i) = \begin{cases}
	d, & c_1(e) = d, \\
	i -\kappa(e_-) + c_1(e)  , &  \text{otherwise}
	\end{cases}$$
	for every edge $\overline{e} = (e,i) \in E(\overline{G})$. As in the above definition of  $E(\overline{G})$ it is easy to verify that the coloring $\overline{c}$ is independent of the choice of  orientation on $E(G)$.
	
	Consider a vertex $\overline{v} = (v,i) \in V(\overline{G})$. We may assume without loss of generality that every edge $e \in E(G)$ incident at $v$ is oriented so that $e_- = v$. Therefore an edge $\overline{e} = (e,i) \in E(\overline{G})$ is incident at $\overline{v}$ if and only if $e$ is incident at $v$. This shows that $p$ is a covering and in particular that $\overline{G}$ is $d$-regular. The edge $\overline{e} = (e,i)$ with $c_1(e) = d$ is the unique edge incident at $\overline{v}$ with $\overline{c}(\overline{e})= d$. All other edges $\overline{e} = (e,i)$ with $e$ incident at $v$ and $c_1(e) \neq d$ have the value of $i - \kappa(e_-) \in C$ in common. It follows that $\overline{c}$ is a proper coloring. Moreover, the condition $(\overline{G},c_1 \circ p)[d] = (\overline{G},\overline{c})[d]$  clearly holds.
	
	%

	We claim that $\delta(e) = 0$ holds for every $e \in E(G_\text{r})$. This follows by definition whenever $c_1(e) = d$. In the remaining case  $c_1(e) \neq d$ and $c_2(e) = d$. This implies  $e = \rho(e_+) = \rho(e_-)$ so that $c_1(e) = \kappa(e_+) = \kappa(e_-)$ and $\delta(e) = 0$ as well. In addition note that $\overline{c}(e,i) = i$ for every $e \in E(G_\text{r})$ with $c_1(e) \neq d$.
	
	Let $\gamma$ be any cycle in $G_\text{r}$. The previous claim implies that $p^{-1}(\gamma)$ is a disjoint union of the  $\overline{c}$-bi-chromatic  cycles $\gamma \times \{i\}$ of type $i,d$  for every $i \in C$. \end{proof}

\begin{remark}
	The need to choose orientation in the proof above can be removed by using Serre's formalism for graphs \cite{serre1980trees}.
\end{remark}

\begin{cor}
\label{cor:cover aligning one color}
 There is a finite graph covering $p : \overline{G} \to G$ and a proper coloring $\overline{c}'$ of $\overline{G}$ so that $(\overline{G},c_1 \circ p) \sim (\overline{G}, \overline{c}')$ and $(\overline{G},\overline{c}')[d] = (\overline{G},c_2 \circ p)[d]$.
\end{cor}
\begin{proof}
Let   $p:\overline{G}\to G$ be a finite graph covering and  $\overline{c}$ a proper  coloring of  $\overline{G}$ as constructed in Lemma \ref{lemma: cool cover} above. In particular, the subset of all edges $\overline{e} \in E(\overline{G})$ with $\overline{c}(\overline{e}) = d$ and $c_2\circ p(\overline{e}) \neq d$ or alternatively with $\overline{c}(\overline{e}) \neq d$ and $c_2\circ p(\overline{e}) = d$ is  precisely $p^{-1}(G_\text{r})$, and this is a disjoint union of $\overline{c}$-bi-chromatic cycles. Let $\overline{c}'$ be the coloring of $\overline{G}$ obtained from $\overline{c}$ by performing the corresponding edge Kempe switches, so that $(\overline{G},c_1 \circ p) \sim (\overline{G}, \overline{c}')$ and $(\overline{G},\overline{c}')[d] = (\overline{G},c_2 \circ p)[d]$ as required.
\end{proof}

\begin{proof}[Proof of Theorem \ref{thm:main result}]
We wish to prove that $\mathcal{K}_d$ is true for every natural number $d$. Clearly $\mathcal{K}_1$ and $\mathcal{K}_2$ hold. Assume $\mathcal{K}_{d-1}$ and let $G$ be any finite $d$-regular graph admitting two proper edge colorings $c_1$ and $c_2$. 
According to Corollary \ref{cor:cover aligning one color} there is a finite cover $p : \overline{G} \to G$ and a coloring $\overline{c}'$ of $\overline{G}$ so that $(\overline{G},c_1 \circ p) \sim (\overline{G}, \overline{c}')$ and $(\overline{G},\overline{c}')[d] = (\overline{G},c_2 \circ p)[d]$. 
Relying twice  on   Lemmas \ref{lemma:Kempe equivalence preserved under covers} and \ref{prop: end of proof}  we may pass to  a further finite cover $r : \overline{\overline{G}} \to \overline{G}$ and conclude that
 $$(\overline{\overline{G}}, c_1 \circ p \circ r) \sim (\overline{\overline{G}},\overline{c}' \circ r)\sim (\overline{\overline{G}}, c_2\circ p \circ r).$$
 
To establish  $\mathcal{K}_d$ it remains to verify that the degree of the covering $r : \overline{\overline{G}} \to G$ is bounded from above by a constant $\beta(d)$ depending only on $d$. It is clear that $\beta(1) = \beta(2) = 1$. We claim that $\beta$ satisfies the recursion formula $$\beta(d) = (d-1) \beta(d-1)^2.$$
 The degree of the covering $p$ constructed explicitly  in  Lemma \ref{lemma: cool cover}   is precisely $d-1$. Note that we pass to a further cover twice when relying on  Lemma \ref{prop: end of proof} and   the covering degree increases   by a factor of $\beta(d-1)$ each time. As explained in Remark \ref{rem:about using lemma} no further covers are necessary for the proof. This establishes the claim.
\end{proof}

\subsection*{An example}

Let $G = K_{3,3}$ denote the complete bipartite graph on six vertices. The graph $G$ is $3$-regular. It is known \cite{haas2014counting} that $G$ admits two edge-Kempe inequivalent colorings $c_1$ and $c_2$.  These are   illustrated in the bottom row of Figure \ref{fig:example}. The colors $1,2$ and $3$  correspond to blue, red and black, respectively. 

The required  graph covering $\overline{G}$ and edge-Kempe switches are described in the top row of Figure \ref{fig:example}. These are performed along  the bold cycles and indicated by the $\leadsto$ sign. The value of the function $\kappa : V(G) \to C=\mathbb{Z} / 2\mathbb{Z}=\{{\color{blue}1},{\color{red} 2}=0\}$ is indicated on the vertices of $(G,c_1)$ in the left bottom graph.
%
%
\vspace{25pt}

\begin{center}

\begin{figure}[h]
		\def\svgwidth{\textwidth}
		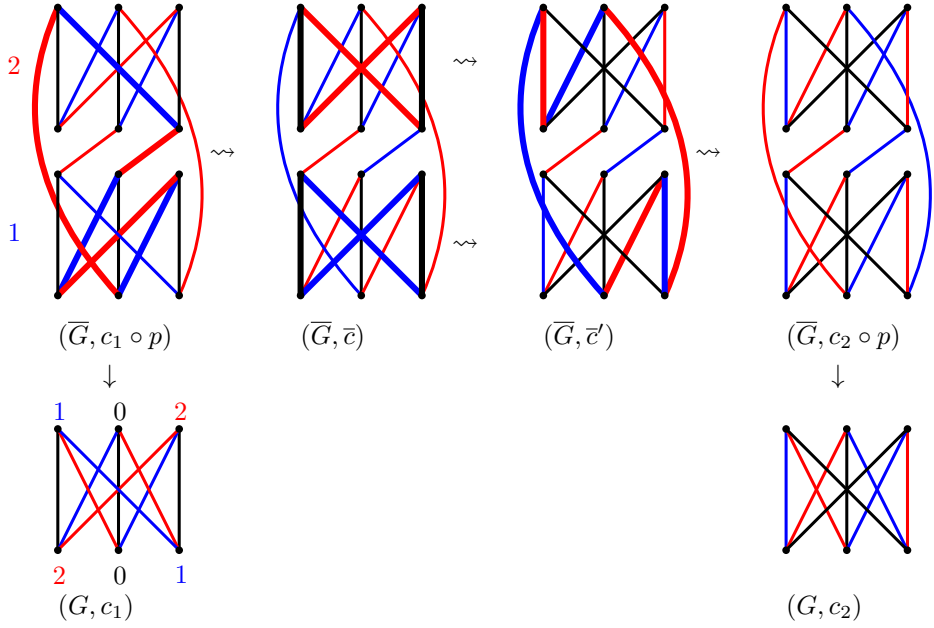
		\caption{A demonstration of the proof of Theorem \ref{thm:main result} for two edge-Kempe inequivalent colorings $c_1$ and $c_2$ of the graph $K_{3,3}$.}

	\label{fig:example}
\end{figure}
\end{center}


%
%


\bibliographystyle{alpha}
\bibliography{recolorings}

\end{document}